\documentclass[amstex,12pt, amssymb]{article}

\usepackage{mathtext}
\usepackage[cp1251]{inputenc}
\usepackage[T2A]{fontenc}
\usepackage[dvips]{graphicx}
\usepackage{amsmath}
\usepackage{amssymb}
\usepackage{amsxtra}
\usepackage{latexsym}
\usepackage{ifthen}

\newcounter{lemma}[section]

\newcounter{corollary}[section]

\newcounter{remark}[section]

\newcounter{theorem}[section]

\newcounter{proposition}[section]

\numberwithin{equation}{section}

\begin{document}

\markboth{\centerline{E. SEVOST'YANOV}} {\centerline{ON LOCAL AND
BOUNDARY BEHAVIOR}}

\def\cc{\setcounter{equation}{0}
\setcounter{figure}{0}\setcounter{table}{0}}

\overfullrule=0pt


\author{{E. SEVOST'YANOV}\\}

\title{
{\bf ON LOCAL AND BOUNDARY BEHAVIOR OF MAPPINGS IN METRIC SPACES}}

\date{\today}
\maketitle

\begin{abstract} Open discrete mappings with a
modulus condition in metric spaces are considered. Some results
related to local behavior of mappings as well as theorems about
continuous extension to a boundary are proved.
\end{abstract}


\section{Introduction}

The paper is devoted to the study of quasiregular mappings and their
natural generalizations investigated long time, see e.g. \cite{AC},
\cite{Cr$_1$}--\cite{Cr$_2$}, \cite{Gol$_1$}--\cite{Gol$_2$},
\cite{GRSY}, \cite{IM}, 
\cite{MRSY}, \cite{MRV$_1$}--\cite{MRV$_3$}, \cite{Re}, \cite{Ri},
\cite{Vu} and further references therein. We also refer to work of
Novosibirsk mathematical school, see \cite{Vo$_1$}--\cite{UV}.

\medskip
As known, boundary and local behavior of quasiregular mappings in
${\Bbb R}^n$ are the main subjects of investigation in many works,
see \cite{Ge}, 
\cite{Na}, \cite{MRV$_2$}, \cite{Ri},
\cite{Va$_1$}, \cite{Va$_2$} etc. It should also be noted a large
number of works by mapping with finite distortion in this context,
see e.g. \cite{Cr$_1$}--\cite{Cr$_2$},
\cite{Gol$_1$}--\cite{Gol$_2$}, \cite{GRSY}, \cite{IM}, \cite{HK},
\cite{MRSY} and \cite{Ra}. Besides that, we refer to works, where
mappings obeying modular inequalities are studied, see
\cite{IR$_1$}--\cite{IR$_2$}, \cite{RS}, \cite{Sm} and
\cite{Sev$_1$}--\cite{Sev$_3$}. Mappings mentioned above are called
$Q$-mappings, and were introduced by O.~Martio together with
V.~Ryazanov, U.~Srebro and E.~Yakubov, see \cite{MRSY}.

\medskip
Now we return to \cite{Sev$_1$}--\cite{Sev$_3$}. Local behavior of
mappings satisfying modular inequalities is studied in
\cite{Sev$_1$}. In particular, we have proved here that a family of
mappings mentioned above is equicontinuous provided that
characteristic of quasiconformality $Q(x)$ has a finite mean
oscillation at the corresponding point. In \cite{Sev$_2$}, we have
proved that sets of zero modulus with weight $Q$ (in particular,
isolated singularities) are removable for discrete open $Q$-mappings
if the function $Q(x)$ has finite mean oscillation or a logarithmic
singularity of order not exceeding $n-1$ on the corresponding set.
The problem of extension of mappings $f:D\rightarrow {\Bbb R}^n$
with modular condition to the boundary of a domain $D$ has been
investigated in \cite{Sev$_3$}. Under certain conditions imposed on
a measurable function $Q(x)$ and the boundaries of the domains $D$
and $D^{\,\prime}=f(D)$ we show that an open discrete mapping
$f:D\rightarrow {\Bbb R}^n$ with quasiconformality characteristic
$Q(x)$ can be extended to the boundary $\partial D$ by continuity.

\medskip
Now we continue studying mappings satisfying modular conditions. In
the present paper we show that some results from
\cite{Sev$_1$}--\cite{Sev$_3$} holds not only in ${\Bbb R}^n,$ but
in metric spaces, also. Here we assume that mapping $f$ is not
injective, as rule, however, $f$ is open and discrete. In addition,
we need require the existence of maximal liftings of curves under
mapping $f.$ Note that the openness and discreteness of $f$ in
${\Bbb R}^n$ implies the existence of maximal liftings of curves
(see \cite[Ch.~3.II]{Ri}).

\medskip
\section{On equicontinuity of homeomorphisms between metric spaces}

\medskip
Let us give some definitions. Recall, for a given continuous path
$\gamma:[a, b]\rightarrow X$ in a metric space $(X, d),$ that its
length is the supremum of the sums
$$
\sum\limits^{k}_{i=1} d(\gamma(t_i),\gamma(t_{i-1}))
$$
over all partitions $a=t_0\leqslant t_1\leqslant\ldots\leqslant
t_k=b$ of the interval $ [a,b].$ The path $\gamma$ is called {\it
rectifiable} if its length is finite.
\medskip

\medskip
Given a family of paths $\Gamma$ in $X$, a Borel function
$\varrho:X\rightarrow[0,\infty]$ is called {\it admissible} for
$\Gamma$, abbr. $\varrho\in {\rm adm}\,\Gamma$, if
\begin{equation}\label{eq13.2}
\int\limits_{\gamma}\varrho\,ds\ \geqslant\ 1\end{equation} for all
(locally rectifiable) $\gamma\in\Gamma$.
\medskip
Everywhere further, for any sets $E, F,$ and $G$ in $X$, we denote
by $\Gamma(E, F, G)$ the family of all continuous curves $\gamma:[0,
1]\rightarrow X$ such that $\gamma(0)\in E,$ $\gamma(1)\in F,$ and
$\gamma(t)\in G$ for all $t\in (0, 1).$ For $x_0\in X$ and $r>0,$
the ball $\{x\in X: d(x, x_0)<r\}$ is denoted by $B(x_0, r).$
Everywhere further $(X, d, \mu)$ and $\left(X^{\,\prime},
d^{\,\prime}, \mu^{\,\prime}\right)$ are metric spaces with metrics
$d$ and $d^{\,\prime}$ and locally finite Borel measures $\mu$ and
$\mu^{\,\prime},$ correspondingly.

\medskip An open set any two points of which can be connected by a
curve is called a domain in $X.$ The modulus of a family of curves
$\Gamma$ in a domain $G$ of finite Hausdorff dimension
$\alpha\geqslant 2$ from $X$ is defined by the equality
\begin{equation}\label{eq13.5}
M_{\,\alpha}(\Gamma)=\inf\limits_{\rho\in {\rm
adm}\,\Gamma}\int\limits_{G}\varrho^{\,\alpha}(x) d\mu(x)\,.
\end{equation} In the case of the path family $\Gamma^{\,\prime}=f(\Gamma)$ we take
the Hausdorff dimension $\alpha^{\,\prime}$ of the domain
$G^{\,\prime}.$

\medskip
A family of paths $\Gamma_1$ in $X$ is said to be {\it minorized} by
a family of paths $\Gamma_2$ in $X$, abbr. $\Gamma_1>\Gamma_2$, if,
for every path $\gamma_1\in\Gamma_1,$ there is a path
$\gamma_2\in\Gamma_2$ such that $\gamma_2$ is a restriction of
$\gamma_1$. In this case
\begin{equation}\label{eq32*A}
\Gamma_1
> \Gamma_2 \quad \Rightarrow \quad M_{\alpha}(\Gamma_1)\le M_{\alpha}(\Gamma_2)
\end{equation} (see \cite[Theorem~1]{Fu}).

\medskip
Let $G$ and $G^{\,\prime}$ be domains with finite Hausdorff
dimensions $\alpha$ and $\alpha^{\,\prime}\geqslant 2$ in spaces
$(X,d,\mu)$ and $(X^{\,\prime},d^{\,\prime}, \mu^{\,\prime}),$ and
let $Q:G\rightarrow[0,\infty]$ be a measurable function. Set $S(x_0,
r_i)=\{x\in X: d(x_0, x)=r_i\},$ $i=1,2,$ $0<r_1<r_2<\infty.$
Following to \cite[Ch.~7]{MRSY}, we say that a mapping
$f:G\rightarrow G^{\,\prime}$ is a ring $Q$-mapping at a point
$x_0\in G$ if the inequality
\begin{equation}\label{eq1}
M_{{\,\alpha}^{\,\prime}}(f(\Gamma(S_1, S_2,
A))\leqslant\int\limits_{A\cap G}Q(x)\eta^{\alpha}(d(x, x_0))d\mu(x)
\end{equation}
holds for any ring
\begin{equation}\label{eq2A}
A=A(x_0, r_1, r_2)=\{x\in X: r_1<d(x, x_0)<r_2\}, \quad 0 < r_1 <
r_2 <\infty\,,
\end{equation}
%
and any measurable function
$\eta:(r_1, r_2)\rightarrow [0, \infty]$ such that
\begin{equation}\label{eq*3!!}
\int\limits_{r_1}^{r_2}\eta(r)dr\geqslant 1\,.
\end{equation}

\medskip
A family $\mathcal F$ of continuous functions $f$ defined on some
metric space $(X, d)$ with values in another metric space $(Y,
d^{\,\prime})$ is called equicontinuous at a point $x_0\in X$ if for
every $\varepsilon > 0$, there exists $\delta > 0$ such that
$d^{\,\prime}(f(x_0), f(x)) < \varepsilon$ for all $f\in \mathcal F$
and all $x$ such that $d(x_0, x) < \delta$. The family is
equicontinuous if it is equicontinuous at each point of $X$. Thus,
by the well-known Ascoli theorem, normality is equivalent to
equicontinuity on compact sets of the mappings in $\mathcal F$.

\medskip Let $G$ be a domain in a space $(X,d,\mu)$. Similarly to
\cite{IR$_1$}, we say that a function $\varphi:G\rightarrow{\Bbb R}$
has {\it finite mean oscillation at a point $x_{0}\in\overline{G}$},
abbr. $\varphi \in FMO(x_{0})$, if
\begin{equation}\label{eq13.4.111} \overline{\lim\limits_{\varepsilon\rightarrow
0}}\,\, \,\frac{1}{\mu(B(x_{0},\varepsilon))}
\int\limits_{B(x_{0},\varepsilon)}|\varphi(x)-\overline{\varphi}_{\varepsilon}|\,\,d\mu(x)<\infty
\end{equation}
where $$\overline{\varphi}_{\varepsilon}
=\frac{1}{\mu(B(x_{0},\varepsilon))}
\int\limits_{B(x_{0},\varepsilon)}\varphi(x)\,\,d\mu(x)$$ is the
mean value of the function $\varphi(x)$ over the set
$$B(x_{0},\varepsilon)=\{x\in G: d(x,x_0)<\varepsilon\}$$ with
respect to the measure $\mu$. Here the condition (\ref{eq13.4.111})
includes the assumption that $\varphi$ is integrable with respect to
the measure $\mu$ over the set $B(x_0,\varepsilon)$ for some
$\varepsilon>0$.

\medskip
Following \cite[section~7.22]{He}, given a real-valued function $u$
in a metric space $X,$ a Borel function $\rho\colon X\rightarrow [0,
\infty]$ is said to be an {\it upper gradient} of a function
$u:X\to{\Bbb R}$ if $|u(x)-u(y)|\leqslant
\int\limits_{\gamma}\rho\,|dx|$ for each rectifiable curve $\gamma$
joining $x$ and $y$ in $X.$ Let $(X, \mu)$ be a metric measure space
and let $1\leqslant p<\infty.$ We say that $X$ admits {\it a $(1;
p)$-Poincare inequality} if there is a constant $C\geqslant 1$ such
that
$$\frac{1}{\mu(B)}\int\limits_{B}|u-u_B|d\mu(x)\leqslant C\cdot({\rm
diam\,}B)\left(\frac{1}{\mu(B)} \int\limits_{B}\rho^p
d\mu(x)\right)^{1/p}$$
for all balls $B$ in $X,$ for all bounded continuous functions $u$
on $B,$ and for all upper gradients $\rho$ of $u.$ Metric measure
spaces where the inequalities
$$\frac{1}{C}R^{n}\leqslant \mu(B(x_0,
R))\leqslant CR^{n}$$
hold for a constant $C\geqslant 1$, every $x_0\in X$  and all
$R<{\rm diam}\,X$, are called {\it Ahlfors $n$-regular.} As known,
Ahlfors $n$-regular spaces have Hausdorff dimension  $\alpha$ (see
e.g. \cite[p.~61--62]{He}). A domain $G$ in a topological space $T$
is called {\it locally connected at a point} $x_0\in\partial G$ if,
for every neighborhood $U$ of the point $x_0,$ there is its
neighborhood $V\subset U$ such that $V\cap G$ is connected (see
\cite[I.6, $\S\,49$]{Ku}).

\medskip
\begin{theorem}\label{theor4*!} {\sl\,Let $G$ be a domain in a locally connected
and a locally compact metric space $(X, d, \mu)$ with a finite
Hausdorff dimension $\alpha\geqslant 2,$ and let $(X^{\,\prime},
d^{\,\prime}, \mu^{\,\prime})$ be an Ahlfors
$\alpha^{\,\prime}$-regular metric space which supports
$(1;\alpha^{\,\prime})$-Poincare inequality. Let $B_R\subset
X^{\,\prime}$ is a fixed ball of a radius $R.$ Denote
$\frak{R}_{x_0, Q, B_R, \delta}(G)$ a family of ring
$Q$-homeomorphisms $f\colon G\rightarrow B_R\setminus K_f$ at
$x_0\in G$ with $\sup\limits_{x, y\in K_f} d^{\,\prime}(x,
y)\geqslant \delta>0,$ where $K_f\subset B_R$ is some continuum.
Then $\frak{R}_{x_0, Q, B_R, \delta}(G)$ is equicontinuous at
$x_0\in G$ whenever $Q\in FMO(x_0).$}
\end{theorem}

\medskip
The following lemma can be useful under investigations related to
equicontinuity of families of mappings.

\medskip
\begin{lemma}\label{lem4}{\sl\,
Let $G$ be a domain in a metric space $(X, d, \mu)$ with a finite
Hausdorff dimension $\alpha\geqslant 2,$ and let $(X^{\,\prime},
d^{\,\prime}, \mu^{\,\prime})$ be a metric space with a finite
Hausdorff dimension $\alpha^{\,\prime}\geqslant 2.$ Let $f\colon
G\rightarrow X^{\,\prime}$ be a ring $Q$-mapping at $x_0\in G,$ and
let $\varepsilon_0>0$ be such that $\overline{B(x_0,
\varepsilon_0)}\subset G.$ Assume that
 \begin{equation} \label{eq3.7B}
\int\limits_{\varepsilon<d(x, x_0)<\varepsilon_0}
Q(x)\cdot\psi_{\varepsilon}^{\alpha}(d(x, x_0))\, d\mu(x)\leqslant
F(\varepsilon, \varepsilon_0)\qquad\forall\,\,\varepsilon\in(0,
\varepsilon_0^{\,\prime}),
 \end{equation}
for some $\varepsilon_0^{\,\prime}\in (0, \varepsilon_0)$ and some
family of nonnegative Lebesgue measurable functions
$\{\psi_{\varepsilon}(t)\},$ $\psi_{\varepsilon}\colon (\varepsilon,
\varepsilon_0)\rightarrow [0, \infty],$ $\varepsilon\in\left(0,
\varepsilon_0^{\,\prime}\right),$ where
где $F(\varepsilon, \varepsilon_0)$ is some function, and
\begin{equation}\label{eq3AB} 0<I(\varepsilon, \varepsilon_0):=
\int\limits_{\varepsilon}^{\varepsilon_0}\psi_{\varepsilon}(t)dt <
\infty\qquad\forall\,\,\varepsilon\in(0,
\varepsilon_0^{\,\prime}).\end{equation}
Denote $S_1=S(x_0, \varepsilon),$ $S_2=S(x_0, \varepsilon_0)$ and
$A=\{x\in G: \varepsilon<d(x, x_0)<\varepsilon_0\}.$ Then
\begin{equation}\label{eq3B}
M_{\alpha^{\,\prime}}(f(\Gamma(S_1, S_2, A))\leqslant
F(\varepsilon,\varepsilon_0)/ I^{\,\alpha}(\varepsilon,
\varepsilon_0)\qquad
\forall\,\,\varepsilon\in\left(0,\,\varepsilon_0^{\,\prime}\right)\,.
\end{equation}
}
\end{lemma}

\begin{proof}
Set
$\eta_{\varepsilon}(t)=\psi_{\varepsilon}(t)/I(\varepsilon,
\varepsilon_0 ),$ $t\in(\varepsilon,\, \varepsilon_0).$ We observe
that
$\int\limits_{\varepsilon}^{\varepsilon_0}\eta_{\varepsilon}(t)\,dt=1$
for $\varepsilon\in (0, \varepsilon_0^{\,\prime}).$
Now, from the definition of ring $Q$-mapping at $x_0,$ and from
(\ref{eq3.7B}), we obtain (\ref{eq3B}).~$\Box$
 \end{proof}

\medskip
The following statement holds (see~\cite[Proposition~4.7]{AS}).

\medskip
\begin{proposition}\label{pr2}
{\sl Let $X$ be a $\alpha$-Ahlfors regular metric measure space that
supports $(1; \alpha)$-Poincare inequality for some $\alpha> 1.$ Let
$E$ and $F$ be continua contained in a ball $B(x_0, R).$ Then
$$M_{\alpha}(\Gamma(E, F, X))\geqslant \frac{1}{C}\cdot\frac{\min\{{\rm diam}\,E, {\rm diam}\,F\}}{R}$$
for some constant $C>0.$}
\end{proposition}

\medskip
The following lemma provides the main tool for establishing
equicontinuity in the most general situation.

\medskip
\begin{lemma}\label{lem1}{\sl\,
Let $G$ be a domain in a locally connected and locally compact
metric space $(X, d, \mu)$ with a finite Hausdorff dimension
$\alpha\geqslant 2,$ and let $(X^{\,\prime}, d^{\,\prime},
\mu^{\,\prime})$ be an Ahlfors $\alpha^{\,\prime}$-regular metric
space which supports $(1;\alpha^{\,\prime})$-Poincare inequality.
Let $r_0>0$ be such that $\overline{B(x_0, \varepsilon_0)}\subset G$
and $0<\varepsilon_0<r_0.$ Assume that, \eqref{eq3.7B} holds for
some $\varepsilon_0^{\,\prime}\in (0, \varepsilon_0),$ and for some
family of nonnegative Lebesgue measurable function
$\{\psi_{\varepsilon}(t)\},$ $\psi_{\varepsilon}\colon (\varepsilon,
\varepsilon_0)\rightarrow [0, \infty],$ $\varepsilon\in\left(0,
\varepsilon_0^{\,\prime}\right),$ where $F(\varepsilon,
\varepsilon_0)$ is some function for which $F(\varepsilon,
\varepsilon_0)=o(I^{\alpha}(\varepsilon, \varepsilon_0)),$ and
$I(\varepsilon, \varepsilon_0)$ is defined in \eqref{eq3AB}.

Let $B_R\subset X^{\,\prime}$ be a fixed ball of a radius $R.$
Denote $\frak{R}_{x_0, Q, B_R, \delta}(G)$ a family of all ring
$Q$-homeomorphisms $f\colon G\rightarrow B_R\setminus K_f$ at
$x_0\in G$ with $\sup\limits_{x, y\in K_f} d^{\,\prime}(x,
y)\geqslant \delta>0,$ where $K_f\subset B_R$ is a fixed continuum.
Now, $\frak{R}_{Q, x_0, B_R, \delta}(G)$ is equicontinuous at $x_0.$
}
 \end{lemma}

\medskip
\begin{proof}
Fix $x_0\in G,$ $f\in\frak{R}_{x_0, Q, B_R, \delta}(G).$ Since $X$
is locally connected and locally compact space, we can find a
sequence $B(x_0, \varepsilon_k),$ $k=0,1,2,\ldots,$
$\varepsilon_k\rightarrow 0$ as $k\rightarrow\infty,$ such that
$V_{k+1}\subset \overline{B(x_0, \varepsilon_k)}\subset V_k,$ where
$V_k$ are continua in $G.$ Observe that $f(V_k)$ are $K_f$ continua
in $B_R,$ in fact, $f(V_k)$ is a continuum as continuous image of a
continuum (see e.g. \cite[Theorem~1.III.41 and Theorem 3.I.46]{Ku}).
Now, by Proposition \ref{pr2} we obtain that
\begin{equation}\label{eq2}
M_{\alpha^{\,\prime}}(K_f, f(V_k), X^{\,\prime}))\geqslant
\frac{1}{C}\cdot\frac{\min\{{\rm diam}\,K_f, {\rm diam}\,
f(V_k)\}}{R}
 \end{equation}
at some $C>0.$ Note that $\gamma\in \Gamma(K_f, f(V_k),
X^{\,\prime})$ does not fully belong to $f(B(x_0, \varepsilon_0))$
as well as $X^{\,\prime}\setminus f(B(x_0, \varepsilon_0)),$ so
there exists $y_1\in |\gamma|\cap f(S(x_0, \varepsilon_0))$ (see
\cite[Theorem~1, $\S\,46,$ section.~I]{Ku}). Let $\gamma:[0,
1]\rightarrow X^{\,\prime}$ and $t_1\in (0, 1)$ be such that
$\gamma(t_1)=y_1.$ Without loss of generalization, we can consider
that $|\gamma|_{[0, t_1)}|\in f(B(x_0, \varepsilon_0)).$ Denote
$\gamma_1:=\gamma|_{[0, t_1)},$ and set
$\alpha_1=f^{\,-1}(\gamma_1).$ Observe that $|\alpha_1|\in B(x_0,
\varepsilon_0).$ Moreover, note that $\alpha_1$ does not wholly
belong to $\overline{B(x_0, \varepsilon_{k-1}}$ as well as to
$X\setminus\overline{B(x_0, \varepsilon_{k-1})}.$ Thus, there exists
$t_2\in (0, t_1)$ with $\alpha_1(t_2)\in S(x_0, \varepsilon_{k-1})$
(see \cite[Theorem~1, $\S\,46,$ Section.~I]{Ku}). Without loss of
generality, we can consider that $|\alpha_{[t_2, t_1]}|\in
X\setminus\overline{B(x_0, \varepsilon_{k-1})}.$ Set
$\alpha_2=\alpha_1|_{[t_2, t_1]}.$ Observe that
$\gamma_2:=f(\alpha_2)$ is a subcurve of $\gamma.$ From saying
above,
$$\Gamma(K_f, f(V_k),
X^{\,\prime})>\Gamma(f(S(x_0, \varepsilon_{k-1})), f(S(x_0,
\varepsilon_0)), f(A)))\,,$$
where $A=\{x\in X: \varepsilon_{k-1}<d(x, x_0)<\varepsilon_0\},$
whence by (\ref{eq32*A})
\begin{equation}\label{eq3}
M_{\alpha^{\,\prime}}(\Gamma(K_f, f(V_k), X^{\,\prime}))\leqslant
M_{\alpha^{\,\prime}}(\Gamma(f(S(x_0, \varepsilon_{k-1})), f(S(x_0,
\varepsilon_0)), f(A)))\,.
\end{equation}
By (\ref{eq2}) and (\ref{eq3}), we conclude that
\begin{equation}\label{eq4}
M_{\alpha^{\,\prime}}(\Gamma(f(S(x_0, \varepsilon_{k-1})), f(S(x_0,
\varepsilon_0)), f(A)))\geqslant \frac{1}{C}\cdot\frac{\min\{{\rm
diam}\,K_f, {\rm diam}\, f(V_k)\}}{R}\,.
\end{equation}
From other hand, by Lemma \ref{lem4} and by $F(\varepsilon,
\varepsilon_0)=o(I^{\alpha}(\varepsilon, \varepsilon_0)),$ it
follows that
$$M_{\alpha^{\,\prime}}(\Gamma(f(S(x_0, \varepsilon_{k-1})), f(S(x_0, \varepsilon_0)),
f(A)))\rightarrow 0$$
as $k\rightarrow \infty.$ Therefore, for every $\sigma>0$ there
exists $k_0\in {\Bbb N}=k_0(\sigma)$ such that
$$M_{\alpha^{\,\prime}}(\Gamma(f(S(x_0, \varepsilon_{k-1})), f(S(x_0, \varepsilon_0)),
f(A)))<\sigma$$
for every $k\geqslant k_0.$ Now, by (\ref{eq4}), it follows that
\begin{equation}\label{eq5}
\min\{{\rm diam}\,K_f, {\rm diam}\, f(V_k)\}< \sigma
\end{equation}
for $k\geqslant k_0.$ Since ${\rm diam}\,K_f\geqslant \delta>0$ for
every $f$, we obtain that
$$\min\{{\rm diam}\,K_f, {\rm diam}\, f(V_k)\}={\rm diam}\, f(V_k)$$
for every $k\geqslant k_1(\sigma).$ Now, by (\ref{eq5})
\begin{equation}\label{eq1C}
{\rm diam}\, f(V_k)<\sigma
\end{equation}
for every $k\geqslant k_1(\sigma).$ Since $V_{k+1}\subset
\overline{B(x_0, \varepsilon_k)}\subset V_k,$ the inequality
(\ref{eq1C}) holds in $\overline{B(x_0, \varepsilon_k)}$ as
$k\geqslant k_1(\sigma).$ Set
$\varepsilon(\sigma):=\varepsilon_{k_1}.$ Finally, given $\sigma>0$
there exists $\varepsilon(\sigma)>0$ such that $d^{\,\prime}(f(x),
f(x_0))<\sigma$ as $d(x, x_0)<\varepsilon(\sigma).$ So,
$\frak{R}_{Q, x_0, B_R, \delta}(G)$ is equicontinuous at
$x_0.$~$\Box$
\end{proof}

\medskip
The following statement can be found in \cite[Lemma~4.1]{RS}.

\medskip
\begin{proposition}\label{pr3}
{\sl Let $G$ be a domain Ahlfors $\alpha$-regular metric space $(X,
d, \mu)$ at $\alpha\geqslant 2.$ Assume that $x_0\in \overline{G}$
and $Q:G\rightarrow [0, \infty]$ belongs to $FMO(x_0).$ If
\begin{equation}\label{eq7}
\mu(G\cap B(x_0, 2r))\leqslant
\gamma\cdot\log^{\alpha-2}\frac{1}{r}\cdot \mu(G\cap B(x_0, r))
\end{equation}
for some $r_0>0$ and every $r\in (0, r_0),$ then $Q$
satisfies~\eqref{eq3.7B}{\em,} where $G(\varepsilon):=F(\varepsilon,
\varepsilon_0)/I^n(\varepsilon, \varepsilon_0)$ obeying \/{\em:}
$G(\varepsilon)\rightarrow 0$ as $\varepsilon\rightarrow 0,$ and
$\psi_{\varepsilon}(t)\equiv\psi(t):=\frac{1}{t\log\frac{1}{t}}.$}
\end{proposition}

\medskip
{\it Proof of the Theorem~{\em\ref{theor4*!}} follows from
Lemma~{\em\ref{lem1}} and Proposition~{\em\ref{pr3}}}.~$\Box$

\medskip
Taking into account \cite[Corollary~4.1]{RS}, by Lemma~\ref{lem1},
we obtain the following.

\medskip
\begin{corollary}\label{cor1}
{\sl A conclusion of Theorem~{\em\ref{theor4*!}} holds, if instead
of condition $Q\in FMO(x_0)$ we require that
$$\limsup\limits_{\varepsilon\rightarrow 0}\frac{1}{\mu(B(x_0, \varepsilon))}\int\limits
_{B(x_0, \varepsilon)}Q(x)d\mu(x)<\infty\,.$$ }
\end{corollary}

\section{Equicontinuity of open discrete mappings }

In this section we prove a result similar to Theorem~\ref{theor4*!},
where instead of homeomorphisms are considered open discrete
mappings. However, in this case we have to require the following
additional condition: the mapping should have a maximal lifting
relative to an arbitrary curve. To give a definition.

Let $D\subset X,$ $f:D \rightarrow X^{\,\prime}$ be a discrete open
mapping, $\beta: [a,\,b)\rightarrow X^{\,\prime}$ be  a curve, and
$x\in\,f^{-1}\left(\beta(a)\right).$ A curve $\alpha:
[a,\,c)\rightarrow D$ is called a {\it maximal $f$-lifting} of
$\beta$ starting at $x,$ if $(1)\quad \alpha(a)=x\,;$ $(2)\quad
f\circ\alpha=\beta|_{[a,\,c)};$ $(3)$\quad for $c<c^{\prime}\le b,$
there is no curves $\alpha^{\prime}: [a,\,c^{\prime})\rightarrow D$
such that $\alpha=\alpha^{\prime}|_{[a,\,c)}$ and $f\circ
\alpha^{\,\prime}=\beta|_{[a,\,c^{\prime})}.$ In the case
$X=X^{\,\prime}={\Bbb R}^n,$ the assumption on $f$ yields that every
curve $\beta$ with $x\in f^{\,-1}\left(\beta(a)\right)$ has  a
maximal $f$-lif\-ting starting at $x$ (see
\cite[Corollary~II.3.3]{Ri}, \cite[Lemma~3.12]{MRV$_3$}).

\medskip
Consider the condition

\medskip\medskip\medskip
$\textbf{A}:$ {\bf for all $\beta: [a,\,b)\rightarrow X^{\,\prime}$
and $x\in f^{\,-1}\left(\beta(a)\right),$ a mapping $f$ has a
maximal $f$-lif\-ting starting at $x.$}

\medskip\medskip\medskip
Given $x_0\in D$ and $0<\varepsilon<\varepsilon_0,$ let $A=A(x_0,
\varepsilon, \varepsilon_0)$ be defined in (\ref{eq2A}), let
$S_i=S(x_0, r_i)$ be sphere centered at $x_0$ of a radius $r,$ and
let $Q\colon D\rightarrow [0,\infty]$ be a measurable function. The
following lemma holds.

\medskip
 \begin{lemma}\label{lem5}
{\sl\, Let $G$ be a domain in a metric space $(X, d, \mu)$ with a
finite Hausdorff dimension $\alpha\geqslant 2,$ and let
$(X^{\,\prime}, d^{\,\prime}, \mu^{\,\prime})$ be a metric space
which has a finite Hausdorff dimension $\alpha^{\,\prime}\geqslant
2.$ Let $f\colon G\rightarrow X^{\,\prime}$ be a ring $Q$-mapping at
$x_0\in G,$ and let $0<\varepsilon_0<{\rm dist}\,(x_0, \partial D)$
be such that $\overline{B(x_0, \varepsilon_0)}$ is compactum in $D.$

Assume that
\begin{equation} \label{eq3.7C}
\int\limits_{\varepsilon<d(x, x_0)<\varepsilon_0}
Q(x)\cdot\psi_{\varepsilon}^{\alpha}(d(x, x_0)) \ d\mu(x)\leqslant
F(\varepsilon, \varepsilon_0)\qquad\forall\,\,\varepsilon\in(0,
\varepsilon_0^{\,\prime})
\end{equation}
holds for some $\varepsilon_0^{\,\prime}\in (0, \varepsilon_0),$ and
some family of nonnegative Lebesgue measurable functions
$\{\psi_{\varepsilon}(t)\},$ $\psi_{\varepsilon}\colon (\varepsilon,
\varepsilon_0)\rightarrow [0, \infty],$ $\varepsilon\in\left(0,
\varepsilon_0^{\,\prime}\right),$ where
$F(\varepsilon,\varepsilon_0)$ is some function, and (\ref{eq3AB})
holds.
If $f$ satisfies the condition $\textbf{A},$ then
\begin{equation}\label{eq3C}
M_{\alpha^{\,\prime}}(\Gamma(f(\overline{B(x_0, \varepsilon)}),
\partial f(B(x_0, \varepsilon_0)), X^{\,\prime}))\leqslant
F(\varepsilon,\varepsilon_0)/ I^{\,\alpha}(\varepsilon,
\varepsilon_0)\qquad
\forall\,\,\varepsilon\in\left(0,\,\varepsilon_0^{\,\prime}\right)\,.
\end{equation}
}
 \end{lemma}

 \medskip
 \begin{proof}
We can assume that $\Gamma:=\Gamma(f(\overline{B(x_0,
\varepsilon)}),
\partial f(B(x_0, \varepsilon_0)), X^{\,\prime})\ne \varnothing.$

Now $\partial f(B(x_0, \varepsilon_0))\ne \varnothing.$ Let
$\Gamma^{\,*}$ be a family of maximal $f$-liftings of $\Gamma$
started at $\overline{B(x_0, \varepsilon)}.$ Given a curve
$\beta:[0, 1)\rightarrow X^{\,\prime},$ $\beta\in \Gamma,$ we show
that it's maximal lifting $\alpha:[0, c)\rightarrow X$ satisfies the
condition: $d(\alpha(t), S(x_0, \varepsilon_0))\rightarrow 0$ as
$t\rightarrow c-0.$

\medskip
Assume the contrary, i.e., there exists $\beta\colon
[a,\,b)\rightarrow X^{\,\prime}$ from $\Gamma$ for which it's
maximal lifting $\alpha\colon [a,\,c)\rightarrow B(x_0,
\varepsilon_0)$ satisfies the condition $d(|\alpha|, \partial B(x_0,
\varepsilon_0))=\delta_0>0.$ Consider
$$G=\left\{x\in X:\, x=\lim\limits_{k\rightarrow\,\infty}
\alpha(t_k)
 \right\}\,,\quad t_k\,\in\,[a,\,c)\,,\quad
 \lim\limits_{k\rightarrow\infty}t_k=c\,.$$ Note that $c\ne b.$ In
fact, assume that $c=b,$ then $|\beta|=f(|\alpha|)$ is compactum in
$B(x_0, \varepsilon_0),$ and we obtain a contradiction.

Now, let $c\ne b.$ Letting to subsequences, if it is need, we can
restrict us by monotone sequences $t_k.$ For $x\in G,$ by continuity
of $f,$
$f\left(\alpha(t_k)\right)\rightarrow\,f(x)$ as
$k\rightarrow\infty,$ where $t_k\in[a,\,c),\,t_k\rightarrow c$ as
$k\rightarrow \infty.$ However,
$f\left(\alpha(t_k)\right)=\beta(t_k)\rightarrow\beta(c)$ as
$k\rightarrow\infty.$ Thus, $f$ is a constant on $G$ in $B(x_0,
\varepsilon_0).$ From other hand, $\overline{\alpha}$ is a compact
set, because $\overline{\alpha}$ is a closed subset of the compact
space $\overline{B(x_0, \varepsilon_0)}$ (see \cite[Theorem~2.II.4,
$\S\,41$]{Ku}). Now, by Cantor condition on the compact
$\overline{\alpha},$ by monotonicity of
$\alpha\left(\left[t_k,\,c\right)\right),$
$$G\,=\,\bigcap\limits_{k\,=\,1}^{\infty}\,\overline{\alpha\left(\left[t_k,\,c\right)\right)}
\ne\varnothing\,,
$$
%
see \cite[1.II.4, $\S\,41$]{Ku}. Now, by \cite[Theorem~5.II.5,
$\S\,47$]{Ku}, $\overline{\alpha}$ is connected. By discreteness of
$f,$ $G$ is a single-point set, and $\alpha\colon
[a,\,c)\rightarrow\,B(x_0, \varepsilon_0)$ extends to a closed curve
$\alpha\colon [a,\,c]\rightarrow K\subset B(x_0, \varepsilon_0),$
and $f\left(\alpha(c)\right)=\beta(c).$ By condition $\textbf{A},$
there exists a new maximal lifting $\alpha^{\,\prime}$ of
$\beta|_{[c,\,b)}$ starting in $\alpha(c).$ Uniting $\alpha$ and
$\alpha^{\,\prime},$ we obtain a new lifting
$\alpha^{\,\prime\prime}$ of $\beta,$ which is defined in $[a,
c^{\prime}),$ \,\,$c^{\,\prime}\,\in\,(c,\,b),$ that contradicts to
''maximality'' of $\alpha.$ Thus, $d(\alpha(t), S(x_0,
\varepsilon_0))\rightarrow 0$ as $t\rightarrow c-0.$

Observe that $\Gamma(f(\overline{B(x_0, \varepsilon)}), \partial
f(B(x_0, \varepsilon_0)), X^{\,\prime})>f(\Gamma^{*}),$ and,
consequently, by~\eqref{eq32*A}
\begin{equation}\label{eq7AB}
M_{\alpha^{\,\prime}}\left(\Gamma(f(\overline{B(x_0, \varepsilon)}),
\partial f(B(x_0, \varepsilon_0)), X^{\,\prime}))\right)\leqslant
M_{\alpha^{\,\prime}}\left(f(\Gamma^{*})\right)\,.
\end{equation}
Consider
$$S_{\,\varepsilon}=S(x_0,\,\varepsilon)\,,\quad
S_{\,\varepsilon_{0}}=S(x_0,\,\varepsilon_0)\,,$$
where $\varepsilon_0$ is from conditions of the lemma, and
$\varepsilon\in\left(0,\,\varepsilon_0^{\,\prime}\right).$ Since
every curve $\alpha\in\Gamma^{*}$ satisfies the condition
$d(\alpha(t), S(x_0, \varepsilon_0))\rightarrow 0$ as $t\rightarrow
c-0,$ we obtain that $\Gamma\left(S_{\varepsilon},
S_{\varepsilon_0-\delta}, A(x_0, \varepsilon,
\varepsilon_0-\delta)\right)<\Gamma^{*}$ at sufficiently small
$\delta>0$ and, consequently, $f(\Gamma\left(S_{\varepsilon},
S_{\varepsilon_0-\delta}, A(x_0, \varepsilon,
\varepsilon_0-\delta)\right))<f(\Gamma^{\,*}).$
Now
\begin{equation}\label{eq3.3.1}
M_{\alpha^{\,\prime}}\left(f(\Gamma^{*})\right)\leqslant
 M_{\alpha^{\,\prime}}\left(f\left(\Gamma\left(S_{\varepsilon}, S_{\varepsilon_0-\delta},
 A(x_0, \varepsilon, \varepsilon_0-\delta)\right)\right)\right)\,.
\end{equation}
By~\eqref{eq7AB} and \eqref{eq3.3.1},
\begin{equation}\label{eq5aa}
M_{\alpha^{\,\prime}}(\Gamma(f(\overline{B(x_0, \varepsilon)}),
\partial f(B(x_0, \varepsilon_0)), X^{\,\prime}))\leqslant
 M_{\alpha^{\,\prime}}\left(f\left(\Gamma\left(S_{\varepsilon}, S_{\varepsilon_0-\delta},
 A(x_0, \varepsilon, \varepsilon_0-\delta)\right)\right)\right)\,.
\end{equation}
Let $\eta(t)$ be an arbitrary nonnegative Lebesgue measurable
function with the condition
$\int\limits_{\varepsilon}^{\varepsilon_0}\eta(t)dt=1.$ Consider the
family of function
$\eta_{\delta}(t)=\frac{\eta(t)}{\int\limits_{\varepsilon}^{\varepsilon_0-\delta}\eta(t)dt}.$
(Since $\int\limits_{\varepsilon}^{\varepsilon_0}\eta(t)dt=1,$ we
can choose $\delta>0$ such that
$\int\limits_{\varepsilon}^{\varepsilon_0-\delta}\eta(t)dt>0$).
Since
$\int\limits_{\varepsilon}^{\varepsilon_0-\delta}\eta_{\delta}(t)dt=1,$
$$M_{\alpha^{\,\prime}}\left(f\left(\Gamma\left(S_{\varepsilon}\,,
S_{\varepsilon_0-\delta}, A(x_0, \varepsilon,
\varepsilon_0-\delta)\right)\right)\right)\,\leqslant$$
\begin{equation}\label{eq8A}
\leqslant
\frac{1}{\left(\int\limits_{\varepsilon}^{\varepsilon_0-\delta}\eta(t)dt\right)^{\alpha}}
\int\limits_{\varepsilon<d(x,
x_0)<\varepsilon_0}Q(x)\cdot\eta^{\alpha}(d(x, x_0))\, \ d\mu(x)\,.
 \end{equation}
Letting to the limit as $\delta\rightarrow 0,$ by~\eqref{eq5aa}, we
obtain that
$$M_{\alpha^{\,\prime}}\left(f\left(\Gamma\left(S_{\varepsilon}\,,
S_{\varepsilon_0}, A(x_0, \varepsilon,
\varepsilon_0)\right)\right)\right)\,\leqslant
\int\limits_{\varepsilon<d(x,
x_0)<\varepsilon_0}Q(x)\cdot\eta^{\alpha}(d(x, x_0))\, \ d\mu(x)$$
for every nonnegative Lebesgue measurable function $\eta(t)$ with
$\int\limits_{\varepsilon}^{\varepsilon_0}\eta(t)dt=1.$ The desired
conclusion follows now from the lemma~\ref{lem4}.~$\Box$
 \end{proof}

\medskip
Denote $\frak{L}_{x_0, Q, B_R, \delta, \textbf{A}}(D)$ a family of
all open discrete ring $Q$-mappings $f\colon D\rightarrow
B_R\setminus K_f$ at $x_0\in D$ with $\textbf{A}$-condition, where
$B_R\subset X^{\,\prime}$ is some fixed ball of a radius $R,$ and
$K_f$ is some nondegenerate continuum in $B_R$ with$\sup\limits_{x,
y\in K_f} d^{\,\prime}(x, y)\geqslant \delta>0.$ A following
statement is a main tool for a proof of equicontinuity result in a
general situation.

\medskip
 \begin{lemma}\label{lem1B}
{\sl\, Let $D$ be a domain in a locally compact and locally
connected metric space $(X, d, \mu)$ with a finite Hausdorff
dimension $\alpha\geqslant 2,$ and let $(X^{\,\prime}, d^{\,\prime},
\mu^{\,\prime})$ be an Ahlfors $\alpha^{\,\prime}$-regular metric
space which supports $(1;\alpha^{\,\prime})$-Poincare inequality.

Assume also that, \eqref{eq3.7C} holds for some
$\varepsilon_0^{\,\prime}\in (0, \varepsilon_0)$ and some family of
nonnegative Lebesgue measurable functions
$\{\psi_{\varepsilon}(t)\},$ $\psi_{\varepsilon}\colon (\varepsilon,
\varepsilon_0)\rightarrow (0, \infty),$ $\varepsilon\in\left(0,
\varepsilon_0^{\,\prime}\right),$ where $F(\varepsilon,
\varepsilon_0)$ satisfies the condition $F(\varepsilon,
\varepsilon_0)=o(I^n(\varepsilon, \varepsilon_0))$ as
$\varepsilon\rightarrow 0,$ and $I(\varepsilon, \varepsilon_0)$ is
defined by \eqref{eq3AB}.

Now, $\frak{L}_{x_0, Q, B_R, \delta, \textbf{A}}(D)$ is
equicontinuous at $x_0.$ }
 \end{lemma}

\medskip
 \begin{proof}
Fix $f\in\frak{L}_{x_0, Q, B_R, \delta, \textbf{A}}(D).$ Set
$A:=B(x_0, \varepsilon_0)\subset D.$ Since $X$ is locally connected
and locally compact space, we can find a sequence $B(x_0,
\varepsilon_k),$ $k=0,1,2,\ldots,$ $\varepsilon_k\rightarrow 0$ as
$k\rightarrow\infty,$ such that $V_{k+1}\subset \overline{B(x_0,
\varepsilon_k)}\subset V_k,$ where $V_k$ are continua in $G.$
Observe that $f(V_k)$ are $K_f$ continua in $B_R,$ in fact, $f(V_k)$
is a continuum as continuous image of a continuum (see e.g.
\cite[Theorem~1.III.41 and Theorem 3.I.46]{Ku}).

Note that $\Gamma(K_f, f(V_k), X^{\,\prime})>\Gamma(f(V_k),
\partial f(A), X^{\,\prime})$ (see \cite[Theorem~1.I.5, $\S\,46$]{Ku}), so, by (\ref{eq32*A})
\begin{equation}\label{eq9B}
M_{\alpha^{\,\prime}}(\Gamma(f(V_k), \partial f(A),
X^{\,\prime}))\geqslant M_{\alpha^{\,\prime}}(\Gamma(K_f, f(V_k),
X^{\,\prime}))\,.
\end{equation}
By Proposition \ref{pr2}
\begin{equation}\label{eq2B}
M_{\alpha^{\,\prime}}(\Gamma(K_f, f(V_k), X^{\,\prime})\geqslant
\frac{1}{C_1}\cdot\frac{\min\{{\rm diam}\,f(V_k), {\rm
diam}\,K_f\}}{R}\,.
 \end{equation}
By Lemma~\ref{lem5}, $M_{\alpha^{\,\prime}}(\Gamma(K_f, f(V_k),
X^{\,\prime})\rightarrow 0$ as $k\rightarrow \infty$ and, therefore,
by~\eqref{eq3.7C} and \eqref{eq2B} we obtain that
$$\min\{{\rm diam}\,f(V_k), {\rm diam}\,K_f\}={\rm diam}\,f(V_k)$$ as
$k\rightarrow \infty.$ By~\eqref{eq3.7C} and~\eqref{eq2B} it follows
that, for every $\sigma>0$ there exists $k_0=k_0(\sigma)$ such that
\begin{equation}\label{eq1F}
{\rm diam}\,f(C)\leqslant \sigma
\end{equation}
for every $k\geqslant k_0(\sigma).$ Since $V_{k+1}\subset
\overline{B(x_0, \varepsilon_k)}\subset V_k,$ the inequality
(\ref{eq1F}) holds in $\overline{B(x_0, \varepsilon_k)}$ as
$k\geqslant k_0(\sigma).$ Set
$\varepsilon(\sigma):=\varepsilon_{k_0}.$ Finally, given $\sigma>0$
there exists $\varepsilon(\sigma)>0$ such that $d^{\,\prime}(f(x),
f(x_0))<\sigma$ as $d(x, x_0)<\varepsilon(\sigma)$ for every
$f\in\frak{L}_{x_0, Q, B_R, \delta, \textbf{A}}(D).$ So,
$\frak{L}_{x_0, Q, B_R, \delta, \textbf{A}}(D)$ is equicontinuous at
$x_0.$~$\Box$
\end{proof}

\medskip
Denote $\frak{L}_{x_0, Q, B_R, \delta, \textbf{A}}(D)$ a family of
all open discrete ring $Q$-mappings $f\colon D\rightarrow
B_R\setminus K_f$ at $x_0\in D$ with $\textbf{A}$-condition, where
$B_R\subset X^{\,\prime}$ is some fixed ball of a radius $R,$ and
$K_f$ is some nondegenerate continuum in $B_R$ with$\sup\limits_{x,
y\in K_f} d^{\,\prime}(x, y)\geqslant \delta>0.$ Now, from Lemma
\ref{lem1B} and Proposition \ref{pr3}, we obtain the following
statement.

\medskip
\begin{theorem}\label{th1}
{\sl Let $D$ be a domain in a locally compact and locally connected
metric space $(X, d, \mu)$ with a finite Hausdorff dimension
$\alpha\geqslant 2,$ and let $(X^{\,\prime}, d^{\,\prime},
\mu^{\,\prime})$ be an Ahlfors $\alpha^{\,\prime}$-regular metric
space which supports $(1;\alpha^{\,\prime})$-Poincare inequality.

If $Q\in FMO(x_0),$ then $\frak{L}_{x_0, Q, B_R, \delta,
\textbf{A}}(D)$ is equicontinuous at $x_0.$}
\end{theorem}

\medskip
Taking into account \cite[Corollary~4.1]{RS}, by Lemma~\ref{lem1B},
we obtain the following.

\medskip
\begin{corollary}\label{cor2}
{\sl A conclusion of Theorem~{\em\ref{th1}} holds, if instead of
condition $Q\in FMO(x_0)$ we require that
$$\limsup\limits_{\varepsilon\rightarrow 0}\frac{1}{\mu(B(x_0, \varepsilon))}\int\limits
_{B(x_0, \varepsilon)}Q(x)d\mu(x)<\infty\,.$$ }
\end{corollary}

\section{Removability of isolated singularities}

A proof of the following lemma can be given by analogy with
\cite[Lemma~8.1]{RS}.

\medskip
 \begin{lemma}\label{lem3.1!}
{\sl\, Let $D$ be a domain in a 
metric space $(X, d, \mu)$ with a finite Hausdorff dimension
$\alpha\geqslant 2,$ and let $(X^{\,\prime}, d^{\,\prime},
\mu^{\,\prime})$ be an Ahlfors $\alpha^{\,\prime}$-regular metric
space which supports $(1;\alpha^{\,\prime})$-Poincare inequality.
Assume that, there exists $\varepsilon_0>0$ and a Lebesgue
measurable function $\psi(t)\colon(0, \varepsilon_0)\rightarrow
[0,\infty]$ with the following property: for every
$\varepsilon_2\in(0, \varepsilon_0]$ there exists $\varepsilon_1\in
(0, \varepsilon_2],$ such that
\begin{equation} \label{eq5C}
0<I(\varepsilon,
\varepsilon_2):=\int\limits_{\varepsilon}^{\varepsilon _2}\psi(t)dt
< \infty
\end{equation}
for every $\varepsilon\in (0,\varepsilon_1).$ Suppose also, that
\begin{equation} \label{eq4*}
\int\limits_{\varepsilon<d(x,
x_0)<\varepsilon_0}Q(x)\cdot\psi^\alpha(d(x, x_0)) \
dv(x)\,=\,o\left(I^\alpha(\varepsilon, \varepsilon_0)\right)
\end{equation}
holds as $\varepsilon\rightarrow 0.$

Let $\Gamma$ be a family of all curves
$\gamma(t)\colon(0,1)\rightarrow D\setminus\{x_0\}$ obeying
$\gamma(t_k)\rightarrow x_0$ for some $t_k\rightarrow 0,$
$\gamma(t)\not\equiv x_0.$ Then
$M_{\alpha^{\,\prime}}\left(f(\Gamma)\right)=0.$}
 \end{lemma}

\medskip
In particular, \eqref{eq5C} holds provided that $\psi\in
L^1_{loc}(0, \varepsilon_0)$ satisfies the condition $\psi(t)>0$ for
almost every $t\in (0, \varepsilon_0).$

\medskip
 \begin{proof}
Note that
\begin{equation}\label{eq12*}
\Gamma > \bigcup\limits_{i=1}^\infty\,\, \Gamma_i\,,
\end{equation}
where $\Gamma_i$~ is a family of curves
$\alpha_i(t)\colon(0,1)\rightarrow D\setminus\{x_0\}$ such that
$\alpha_i(1)\in \{0<d(x, x_0)=r_i<\varepsilon_0\},$ and $r_i$ is
some sequence with $r_i\rightarrow 0$ as $i\rightarrow \infty,$ and
$\alpha_i(t_k)\rightarrow x_0$ as $k\rightarrow\infty$ for the same
sequence $t_k\rightarrow 0$ as $k\rightarrow\infty.$ Fix $i\geqslant
1.$ By~\eqref{eq5C}, $I(\varepsilon, r_i)>0$ for some
$\varepsilon_1\in (0, r_i]$ and every $\varepsilon\in(0,
\varepsilon_1).$
Now, observe that, for specified $\varepsilon>0,$ the function
$$\eta(t)=\left\{
\begin{array}{rr}
\psi(t)/I(\varepsilon, r_i), &   t\in (\varepsilon,
r_i),\\
0,  &  t\in {\Bbb R}\setminus (\varepsilon, r_i)
\end{array}
\right. $$ satisfies \eqref{eq*3!!} in
$A(x_0, \varepsilon, r_i)=\{x\in X: \varepsilon<d(x, x_0)< r_i \}.$
Since $f$ is a ring $Q$-mapping at $x_0,$ we obtain that
 \begin{multline}\label{eq11*}
M_{\alpha^{\,\prime}}\left(f\left(\Gamma\left(S(x_0,
\varepsilon),\,S(x_0,
r_i),\,A(x_0, \varepsilon, r_i)\right)\right)\right)\leqslant\\
\leqslant \int\limits_{A(x_0, \varepsilon, r_i)} Q(x)\cdot
\eta^\alpha(d(x, x_0))\ d\mu(x)\,\leqslant {\frak F}_i(\varepsilon),
 \end{multline}
where
${\frak F}_i(\varepsilon)=\,\frac{1}{\left(I(\varepsilon,
r_i)\right)^\alpha}\int\limits_{\varepsilon<d(x,
x_0)<\varepsilon_0}\,Q(x)\,\psi^{\alpha}(d(x, x_0))\,d\mu(x).$
By~\eqref{eq4*}, ${\frak F}_i(\varepsilon)\rightarrow 0$ as
$\varepsilon\rightarrow 0.$
Observe that
\begin{equation}\label{eq5*C}
\Gamma_i>\Gamma\left(S(x_0, \varepsilon),\,S(x_0, r_i),\,A(x_0,
\varepsilon, r_i)\right)
\end{equation}
for every $\varepsilon\in (0, \varepsilon_1).$ Thus,
by~\eqref{eq11*} and~\eqref{eq5*C}, we obtain that
\begin{equation}\label{eq6*}
M_{\alpha^{\,\prime}}(f(\Gamma_i))\leqslant {\frak
F}_i(\varepsilon)\rightarrow 0
\end{equation}
for every fixed $i=1,2,\ldots,$ as $\varepsilon\rightarrow 0.$
However, the left-hand side of \eqref{eq6*} does not depend on
$\varepsilon,$ that implies that
$M_{\alpha^{\,\prime}}(f(\Gamma_i))=0.$ Finally, by~\eqref{eq12*}
and subadditivity of modulus (\cite[Theorem~1(b)]{Fu}), we obtain
that $M_{\alpha^{\,\prime}}(f(\Gamma))=0.$~$\Box$
\end{proof}

\medskip
A domain $D$ is called {\it a locally linearly connected at $x_0\in
\partial D,$} if for every neighborhood $U$ of $x_0$ there exists
a ball $B(x_0, r)$ centered at $x_0$ of some radius $r$ in $U$ such
that $B(x_0, r)\cap D$ is linearly connected. The above definition
slightly differs from the standard (see \cite[I.6, $\S\,49$]{Ku}).
The following lemma provides the main tool for establishing
equicontinuity in the most general situation.

\medskip
\begin{lemma}\label{lem4*}{\sl\,
Let $G:=D\setminus\{x_0\}$ be a domain in a locally compact metric
space $(X, d, \mu)$ with a finite Hausdorff dimension
$\alpha\geqslant 2,$ where $G$ is locally linearly connected at
$x_0\in D,$ and let $(X^{\,\prime}, d^{\,\prime}, \mu^{\,\prime})$
be an Ahlfors $\alpha^{\,\prime}$-regular metric space which
supports $(1;\alpha^{\,\prime})$-Poincare inequality.

Assume that, there exists $\varepsilon_0>0$ and a Lebesgue
measurable function $\psi(t)\colon(0, \varepsilon_0)\rightarrow
[0,\infty]$ with the following property: for every
$\varepsilon_2\in(0, \varepsilon_0]$ there exists $\varepsilon_1\in
(0, \varepsilon_2],$ such that (\ref{eq5C}) holds for every
$\varepsilon\in (0,\varepsilon_1).$ Suppose also that, (\ref{eq4*})
holds as $\varepsilon\rightarrow 0.$

Let $B_R$ be a fixed ball in $X^{\,\prime}$ such that
$\overline{B_R}$ is compactum, and let $K$ be a continuum in $B_R.$
If an open discrete ring $Q$-mapping $f\colon
D\setminus\{x_0\}\rightarrow B_R\setminus K$ at $x_0$ satisfies
$\textbf{A}$-condition, then $f$ has a continuous extension to
$x_0.$
}
 \end{lemma}

\medskip
 \begin{proof}
Since $G=D\setminus\{x_0\}$ is locally linearly connected at $x_0\in
D,$ we can consider that $B(x_0, \varepsilon_0)\setminus\{x_0\}$ is
connected. Assume the contrary, namely that the map has no limit at
$x_0.$ Since $\overline{B_R}$ is compactum, the limit set $C(f,
x_0)$ is not empty. Thus, there exist two sequences $x_j$ and и
$x_j^{\,\prime}$ in $B(x_0,
\varepsilon_0)\setminus\left\{x_0\right\},$ $x_j\rightarrow
x_0,\quad x_j^{\,\prime}\rightarrow x_0,$ such that
$d^{\,\prime}\left(f(x_j),\,f(x_j^{\,\prime})\right)\geqslant a>0$
for all $j\in {\Bbb N}.$ Set $r_j=\max{\left\{d(x_j,
x_0),\,d(x_j^{\,\prime}, x_0)\right\}}.$ By locally linearly
connectedness of $G$ at $x_0,$ we can consider that
$\overline{B(x_0, r_j)}\setminus\left\{x_0\right\}$ is linearly
connected. Now, $x_j$ and $x_j^{\,\prime}$ can be joined by a closed
curve $C_j$ in $\overline{B(x_0, r_j)}\setminus\left\{x_0\right\}.$

Set $\Gamma_{f(E_j)}:=\Gamma(f(C_j), K, B_R).$ By Proposition
\ref{pr2}, $\Gamma_{f(E_j)}\ne \varnothing.$ Let $\Gamma_j^{\,*}$ be
the family of all maximal $f$-liftings of $\Gamma_{f(E_j)}$ starting
at $C_j,$ and lying in $B(x_0,
\varepsilon_0)\setminus\left\{x_0\right\}.$ Such the family is
well-defined because $\textbf{A}$ is satisfied.

Arguing as in the proof of Lemma \ref{lem5}, we can show that
\begin{equation}\label{eq33*!}
\Gamma_j^{\,*}\,=\,\Gamma_{E_{j_1}}\cup \Gamma_{E_{j_2}}\,,
\end{equation}
where $\Gamma_{E_{j_1}}$ is a family of all curves
$\alpha(t)\colon[a,\,c)\rightarrow B(x_0,
\varepsilon_0)\setminus\left\{x_0\right\}$ started at $C_j$ for
which $\alpha(t_k)\rightarrow x_0$ as $t_k\rightarrow c-0$ and some
sequence $t_k\in [a,\,c),$ and $\Gamma_{E_{j_2}}$ is a family of all
curves $\alpha(t)\colon[a,\,c)\rightarrow B(x_0,
\varepsilon_0)\setminus\left\{x_0\right\}$ started at $C_j$ for
which ${\rm dist}\left(\alpha(t_k),\partial B(x_0,
\varepsilon_0)\right)\rightarrow 0$ as $t_k\rightarrow c-0$ and some
sequence $t_k\in [a,\,c).$

By~\eqref{eq33*!},
\begin{equation}\label{eq34*!}
M_{\alpha^{\,\prime}}\left(\Gamma_{f(E_j)}\right)\leqslant
M_{\alpha^{\,\prime}}(f(\Gamma_{E_{j_1}}))\,+\,M_{\alpha^{\,\prime}}(f(\Gamma_{E_{j_2}}))\,.
\end{equation}
By Lemma~\ref{lem3.1!},
$M_{\alpha^{\,\prime}}(f(\Gamma_{E_{j_1}}))=0.$

From other hand, we observe that $\Gamma_{E_{j_2}}>\Gamma(S(x_0,
r_j), S(x_0, \varepsilon_0-\frac{1}{m}), A(x_0, r_j,
\varepsilon_0-\frac{1}{m}))$ for sufficiently large $m\in {\Bbb N}.$
Set
$A_{j}=\{x\in X: r_j<d(x, x_0)< \varepsilon_0-\frac{1}{m}\}$ and
$$ \eta_{j}(t)= \left\{
\begin{array}{rr}
\psi(t)/I(r_j, \varepsilon_0-\frac{1}{m}), &   t\in (r_j,\, \varepsilon_0-\frac{1}{m}),\\
0,  &  t\in {\Bbb R} \setminus (r_j,\, \varepsilon_0-\frac{1}{m}).
\end{array}
\right.$$
Now, we have that
$\int\limits_{r_j}^{\varepsilon_0-\frac{1}{m}}\,\eta_{j}(t) dt
=\,\frac{1}{I\left(r_j,
\varepsilon_0-\frac{1}{m}\right)}\int\limits_{r_j}
^{\varepsilon_0-\frac{1}{m}}\,\psi(t)dt=1. $ Now, by definition of
the ring $Q$-mapping at $x_0$ and by~\eqref{eq34*!}, we obtain that
$$M_{\alpha^{\,\prime}}(f(\Gamma_{E_{j}}))\leqslant\,\frac{1}{{I(r_j,
\varepsilon_0-\frac{1}{m})}^\alpha}\int\limits_{r_j<d(x,
x_0)<\varepsilon_0}\,Q(x)\,\psi^{\alpha}(d(x, x_0))\,d\mu(x)\,.
$$
Letting to the limit at $m\rightarrow\infty$ here, we obtain that
$$
M_{\alpha^{\,\prime}}(f(\Gamma_{E_{j}}))\leqslant\, {\mathcal
S}(r_j):=\frac{1}{{I(r_j,
\varepsilon_0)}^\alpha}\int\limits_{r_j<d(x,
x_0)<\varepsilon_0}\,Q(x)\,\psi^{\alpha}(d(x, x_0))\,d\mu(x).
$$
%
By~\eqref{eq4*},  ${\mathcal S}(r_j)\,\rightarrow\, 0$ as
$j\rightarrow \infty,$ and by~\eqref{eq34*!} we obtain that
\begin{equation}\label{eq3D}
M_{\alpha^{\,\prime}}\left(\Gamma_{f(E_j)}\right)\rightarrow
0\,,\qquad j\rightarrow\infty\,.
\end{equation}
From other hand, by Proposition~\ref{pr2}, we obtain that
 \begin{equation}\label{eq2D}
M_{\alpha^{\,\prime}}(\Gamma_{f(E_j)})\geqslant
\frac{1}{C}\cdot\frac{\min\{{\rm diam}\,f(C_j), {\rm
diam}\,K\}}{R}\geqslant\delta>0
 \end{equation}
because
$d^{\,\prime}\left(f(x_j),\,f(x_j^{\,\prime})\right)\geqslant a>0$
for all $j\in {\Bbb N}$ assumption made above. However, \eqref{eq2D}
contradicts with \eqref{eq3D}. The contradiction obtained above
proves the theorem.~$\Box$
 \end{proof}

\medskip
The following statements can be obtained from Lemma \ref{lem4*} and
Proposition \ref{pr3}.

\medskip
 \begin{theorem}\label{th4}
{\sl\, Let $G:=D\setminus\{x_0\}$ be a domain in a locally compact
metric space $(X, d, \mu)$ with a finite Hausdorff dimension
$\alpha\geqslant 2,$ where $G$ is locally linearly connected at
$x_0\in D,$ and let $(X^{\,\prime}, d^{\,\prime}, \mu^{\,\prime})$
be an Ahlfors $\alpha^{\,\prime}$-regular metric space which
supports $(1;\alpha^{\,\prime})$-Poincare inequality.

Let $B_R$ be a fixed ball in $X^{\,\prime}$ such that
$\overline{B_R}$ is compactum, and let $K$ be a continuum in $B_R.$
If an open discrete ring $Q$-mapping $f\colon
D\setminus\{x_0\}\rightarrow B_R\setminus K$ at $x_0$ satisfies
$\textbf{A}$-condition, and $Q:D\rightarrow(0, \infty)$ has $FMO$ at
$x_0,$ then $f$ has a continuous extension to $x_0.$
}
 \end{theorem}

\medskip
Taking into account \cite[Corollary~4.1]{RS}, by Lemma~\ref{lem1B},
we obtain the following.

\medskip
\begin{corollary}\label{cor3}
{\sl A conclusion of Theorem~{\em\ref{th4}} holds, if instead of
condition $Q\in FMO(x_0)$ we require that
$$\limsup\limits_{\varepsilon\rightarrow 0}\frac{1}{\mu(B(x_0, \varepsilon))}\int\limits
_{B(x_0, \varepsilon)}Q(x)d\mu(x)<\infty\,.$$ }
\end{corollary}

\medskip
The following results complement \cite[Theorem~10.2]{RS}.

\medskip
 \begin{theorem}\label{th5}
{\sl\, Let $G:=D\setminus\{x_0\}$ be a domain in a locally compact
metric space $(X, d, \mu)$ with a finite Hausdorff dimension
$\alpha\geqslant 2,$ where $G$ is locally linearly connected at
$x_0\in D,$ and let $(X^{\,\prime}, d^{\,\prime}, \mu^{\,\prime})$
be an Ahlfors $\alpha^{\,\prime}$-regular metric space which
supports $(1;\alpha^{\,\prime})$-Poincare inequality.

Let $B_R$ be a fixed ball in $X^{\,\prime}$ such that
$\overline{B_R}$ is compactum, and let $K$ be a continuum in $B_R.$
If $f\colon D\setminus\{x_0\}\rightarrow B_R\setminus K$ is a ring
$Q$-homeomorphism at $x_0,$ and $Q:D\rightarrow(0, \infty)$ has
$FMO$ at $x_0,$ or $\limsup\limits_{\varepsilon\rightarrow
0}\frac{1}{\mu(B(x_0, \varepsilon))}\int\limits _{B(x_0,
\varepsilon)}Q(x)d\mu(x)<\infty,$ then $f$ has a continuous
extension to $x_0.$
}
 \end{theorem}

\section{Boundary behavior}

Let $G$ and $G^{\,\prime}$ be domains with finite Hausdorff
dimensions $\alpha$ and $\alpha^{\,\prime}\geqslant 1$ in spaces
$(X,d,\mu)$ and $(X^{\,\prime},d^{\,\prime}, \mu^{\,\prime}),$ and
let $Q:G\rightarrow[0,\infty]$ be a measurable function. Following
to \cite{Sm}, we say that a mapping $f:G\rightarrow G^{\,\prime}$ is
a ring $Q$-mapping at a point $x_0\in \partial G$ if the inequality
%
$$M_{{\,\alpha}^{\,\prime}}(f(\Gamma(C_1, C_0,
A))\leqslant\int\limits_{A\cap G}Q(x)\eta^{\alpha}(d(x,
x_0))d\mu(x)$$
%
holds for any ring
%
$$A=A(x_0, r_1, r_2)=\{x\in X: r_1<d(x, x_0)<r_2\}, \quad 0 < r_1 <
r_2 <\infty\,,$$
%
and any two continua $C_0\subset \overline{B(x_0, r_1)},$
$C_1\subset X\setminus B(x_0, r_2),$ and any measurable function
$\eta:(r_1, r_2)\rightarrow [0, \infty]$ such that (\ref{eq*3!!})
holds.

\medskip
We say that the boundary of the domain $G$ is {\it strongly
accessible at a point $x_0\in
\partial G$}, if, for every neighborhood  $U$ of the point $x_0$,
there is a compact set $E\subset G$, a neighborhood $V\subset U$ of
the point $x_0$ and a number $\delta
>0$ such that $$M_{\alpha}(\Gamma(E, F, G))\geqslant \delta$$ for every continuum
$F$ in $G$ intersecting $\partial U$ and $\partial V.$ We say that
the boundary $\partial G$  is {\it strongly accessible}, if the
corresponding property holds at every point of the boundary. The
following lemma holds.

\medskip
\begin{lemma}\label{lem1A} {\sl\, Let $D$ be a domain in a 
metric space $(X, d, \mu)$ with a finite Hausdorff dimension
$\alpha\geqslant 2,$ $\overline{D}$ is a compact, and let
$(X^{\,\prime}, d^{\,\prime}, \mu^{\,\prime})$ be a metric space
with a finite Hausdorff dimension $\alpha^{\,\prime}\geqslant 2.$
Let $f:D\rightarrow X^{\,\prime}$ be an open discrete ring
$Q$-mapping at $b\in
\partial D,$ $f(D)=D^{\,\prime},$ $D$ is locally linearly connected at $b,$ $C(f,
\partial D)\subset \partial D^{\,\prime},$ and $D^{\,\prime}$ is
strongly accessible at least at one point $y\in C(f, b).$ Assume
that
\begin{equation}\label{eq7***}
0<I(\varepsilon,
\varepsilon_0)=\int\limits_{\varepsilon}^{\varepsilon_0}\psi(t)dt <
\infty
\end{equation}
for every $\varepsilon\in(0, \varepsilon_0),$ for some
$\varepsilon_0>0,$ and for some nonnegative Lebesgue measurable
function $\psi(t),$ $\psi:(0, \varepsilon_0)\rightarrow (0,\infty).$
Assume that
\begin{equation}\label{eq5***}
\int\limits_{A(b, \varepsilon, \varepsilon_0)}
Q(x)\cdot\psi^{\,\alpha}(d(x, b))
 \ d\mu(x) =o(I^{\alpha}(\varepsilon, \varepsilon_0))\,,
\end{equation}
where $A:=A(b, \varepsilon, \varepsilon_0)$ is define in
(\ref{eq2A}). If $f$ satisfies $\textbf{A}$-condition, then $C(f,
b)=\{y\}.$ }
\end{lemma}

\medskip
\begin{proof}
Assume the contrary. Now, there exist two sequences $x_i,$
$x_i^{\,\prime}\in D,$ $i=1,2,\ldots,$ obeying $x_i\rightarrow b,$
$x_i^{\,\prime}\rightarrow b$ as $i\rightarrow \infty,$
$f(x_i)\rightarrow y,$ $f(x_i^{\,\prime})\rightarrow y^{\,\prime}$
as $i\rightarrow \infty$ и $y^{\,\prime}\ne y.$ Observe that $y$ and
$y^{\,\prime}\in
\partial D^{\,\prime},$ because $C(f,
\partial D)\subset \partial D^{\,\prime}$ by assumption of Lemma. By a definition
of strong accessibility of a boundary at $y\in
\partial D^{\,\prime},$ for every neighborhood $U$ of $y,$ there exists
a compact $C_0^{\,\prime}\subset D^{\,\prime},$ a neighborhood $V$
of $y,$ $V\subset U,$ and $\delta>0$ such that
\begin{equation}\label{eq1A}
M_{\alpha^{\,\prime}}(\Gamma(C_0^{\,\prime}, F, D^{\,\prime}))\ge
\delta
>0
\end{equation} for every compact
$F,$ intersecting $\partial U$ and $\partial V.$ By the assumption
$C(f,
\partial D)\subset \partial D^{\,\prime},$ $C_0\cap \partial
D=\varnothing$ for $C_0:=f^{\,-1}(C_0^{\,\prime}).$ Without loss of
generalization, $C_0\cap\overline{B(b, \varepsilon_0)}=\varnothing.$
Since $D$ is locally linearly connected at $b,$ we can join $x_i$
and $x_i^{\,\prime}$ by a curve $\gamma_i,$ which lies in
$\overline{B(b, 2^{\,-i})}\cap D.$ Since $f(x_i)\in V$ and
$f(x_i^{\,\prime})\in D\setminus \overline{U}$ for sufficiently
large $i\in {\Bbb N},$ by (\ref{eq1A}), there exists $i_0\in {\Bbb
N}$ such that
\begin{equation}\label{eq2C}
M_{\alpha^{\,\prime}}(\Gamma(C_0^{\,\prime}, f(\gamma_i),
D^{\,\prime}))\ge \delta
>0
\end{equation}
for every $i\ge i_0\in {\Bbb N}.$ Given $i\in {\Bbb N},$ $i\ge i_0,$
consider a family $\Gamma_i^{\,\prime}$ of maximal $f$-liftings
$\alpha_i(t):[a, c)\rightarrow D$ of $\Gamma(C_0^{\,\prime},
f(\gamma_i), D^{\,\prime})$ started at $\gamma_i.$ (Such a family
exists by condition $\textbf{A}$). Since $C(f,
\partial D)\subset \partial D^{\,\prime},$ we conclude that
$\alpha_i(t)\in \Gamma_i^{\,\prime},$ $\gamma_i:[a, c)\rightarrow
D,$ does not tend to the boundary of $D$ as $t\rightarrow c-0.$ Now
$C(\alpha_i(t), c)\subset D.$ Since $\overline{D}$ is a compact,
$C(\alpha_i(t), c)\ne \varnothing.$

Assume that $\alpha_i(t)$ has no limit at $t\rightarrow c-0.$ We
show that $C(\alpha_i(t), c)$ is a continuum in $D.$ In fact,
$C(\alpha_i(t),
c)=\bigcap\limits_{k\,=\,1}^{\infty}\,\overline{\alpha\left(\left[t_k,\,c\right)\right)},$
where $t_k$ is increasing. By Cantor condition on the compact
$\overline{\alpha},$ by monotonicity of
$\alpha\left(\left[t_k,\,c\right)\right),$
$$G\,=\,\bigcap\limits_{k\,=\,1}^{\infty}\,\overline{\alpha\left(\left[t_k,\,c\right)\right)}
\ne\varnothing\,,
$$
%
see \cite[1.II.4, $\S\,41$]{Ku}.
Now, $G$ is connected as an intersection of countable collection of
decreasing continua (see \cite[Therorem~5, \S\,47(II)]{Ku}).

So, $C(\alpha_i(t), c)$ is a continuum in $D.$ By continuity of $f,$
we obtain that $f\equiv const$ on $C(\alpha_i(t), c),$ which
contradicts with discreteness of $f.$

Now, $\exists \lim\limits_{t\rightarrow c-0}\alpha_i(t)=A_i\in D,$
and $c=b.$ Now, we have that $\lim\limits_{t\rightarrow
b-0}\alpha_i(t):=A_i,$ and, simultaneously, by continuity of $f$ in
$D,$
$$f(A_i)=\lim\limits_{t\rightarrow b-0}f(\alpha_i(t))=\lim\limits_{t\rightarrow b-0}
\beta_i(t)=B_i\in C_0^{\,\prime}\,.$$ It follows from the definition
of $C_0$ that $A_i\in C_0.$ We can immerse $C_0$ into some continuum
$C_1\subset D,$ see \cite[Lemma~1]{Sm}. We can consider that
$C_1\cap\overline{B(b, \varepsilon_0)}=\varnothing$ by decreasing of
$\varepsilon_0>0.$ Putting $I(\varepsilon,
\varepsilon_0):=\int\limits_{\varepsilon}^{\varepsilon_0}\psi(t)dt$
we observe that the function
$$\eta(t)=\left\{
\begin{array}{rr}
\psi(t)/I(2^{-i}, \varepsilon_0), &   t\in (2^{-i},
\varepsilon_0),\\
0,  &  t\in {\Bbb R}\setminus (2^{-i}, \varepsilon_0)\,,
\end{array}
\right. $$ satisfies (\ref{eq*3!!}) at $r_1:=2^{-i},$
$r_2:=\varepsilon_0.$ Now, by (\ref{eq7***})--(\ref{eq5***}) and
definition of the ring $Q$-mapping at the boundary point,
\begin{equation}\label{eq11*A}
M_{\alpha^{\,\prime}}\left(f\left(\Gamma_i^{\,\prime}\right)\right)\le
\Delta(i)\,,
\end{equation}
where $\Delta(i)\rightarrow 0$ as $i\rightarrow \infty.$ However,
$\Gamma(C_0^{\,\prime}, F, D^{\,\prime})=f(\Gamma_i^{\,\prime}),$
and by (\ref{eq11*A}) we obtain that
\begin{equation}\label{eq3F}
M_{\alpha^{\,\prime}}(\Gamma(C_0^{\,\prime}, F, D^{\,\prime}))=
M_{\alpha^{\,\prime}}\left(f(\Gamma_i^{\,\prime})\right)\le
\Delta(i)\rightarrow 0
\end{equation}
as $i\rightarrow \infty.$ However, (\ref{eq3F}) contradicts with
(\ref{eq2C}). Lemma is proved.~$\Box$
\end{proof}

\medskip
The following statements can be obtained from Lemma \ref{lem1A},
Proposition \ref{pr3} and \cite[Corollary~4.1]{RS}.

\medskip
\begin{theorem}\label{th2} {\sl\, Let $D$ be a domain in a
metric space $(X, d, \mu)$ with locally finite Borel measure $\mu$
and finite Hausdorff dimension $\alpha\geqslant 2,$ $\overline{D}$
is a compact, and let $(X^{\,\prime}, d^{\,\prime}, \mu^{\,\prime})$
be a metric space with locally finite Borel measure $\mu^{\,\prime}$
and finite Hausdorff dimension $\alpha^{\,\prime}\geqslant 2.$ Let
$f:D\rightarrow X^{\,\prime}$ be an open discrete ring $Q$-mapping
at $b\in
\partial D,$ $f(D)=D^{\,\prime},$ $D$ is locally linearly connected at $b,$ $C(f,
\partial D)\subset \partial D^{\,\prime},$ and $D^{\,\prime}$ is
strongly accessible at least at one point $y\in C(f, b).$ Assume
that $Q\in FMO(b)$ and, simultaneously, $Q$ obeying (\ref{eq7}) at
$b.$ If $f$ satisfies $\textbf{A}$-condition, then $C(f, b)=\{y\}.$
}
\end{theorem}

\medskip
\begin{theorem}\label{th3} {\sl\, Let $D$ be a domain in a 
metric space $(X, d, \mu)$ with locally finite Borel measure $\mu$
and finite Hausdorff dimension $\alpha\geqslant 2,$ $\overline{D}$
is a compact, and let $(X^{\,\prime}, d^{\,\prime}, \mu^{\,\prime})$
be a metric space with locally finite Borel measure $\mu^{\,\prime}$
and finite Hausdorff dimension $\alpha^{\,\prime}\geqslant 2.$ Let
$f:D\rightarrow X^{\,\prime}$ be an open discrete ring $Q$-mapping
at $b\in
\partial D,$ $f(D)=D^{\,\prime},$ $D$ is locally linearly connected at $b,$ $C(f,
\partial D)\subset \partial D^{\,\prime},$ and $D^{\,\prime}$ is
strongly accessible at least at one point $y\in C(f, b).$ Assume
that $\limsup\limits_{\varepsilon\rightarrow 0}\frac{1}{\mu(B(x_0,
\varepsilon))}\int\limits _{B(b, \varepsilon)}Q(x)d\mu(x)<\infty$
and, simultaneously, $Q$ obeying (\ref{eq7}) at $b.$ If $f$
satisfies $\textbf{A}$-condition, then $C(f, b)=\{y\}.$ }
\end{theorem}

\section{Examples and open problems}

{\bf Example 1.} Now, let us to show that, the $FMO$ condition can
not be replaced by a weaker requirement $Q\in L^p,$ $p\geqslant 1,$
in Theorem \ref{th4} (see \cite[Proposition~6.3]{MRSY}). For
simplicity, we consider a case $X=X^{\,\prime}={\Bbb R}^n.$

\begin{theorem}\label{theor14*!}
{\sl\, Given $p>1,$ there exists $Q\in L^p({\Bbb B}^n),$ $n\ge 2,$
and bounded ring $Q$-homeomorphism $f:{\Bbb
B}^n\setminus\{0\}\rightarrow {\Bbb R}^n$ at $0,$ for which $x_0=0$
is essential singularity.}
\end{theorem}

\medskip
\begin{proof} Set
$$f(x)=\frac{1+|x|^{\alpha}}{|x|}\cdot x\,,$$
where $\alpha\in (0, n/p(n-1)).$ Without loss of generality, we can
consider that $\alpha<1.$ Observe that, $f$ maps ${\Bbb
B}^n\setminus\{0\}$ onto $\{1<|y|<2\}$ in ${\Bbb R}^n,$ and $C(0,
f)= {\Bbb S}^{n-1}.$ Thus, $x_0=0$ is essential singularity.

Now, we show that $f$ is a ring $Q$-homeomorphism at 0 and some
$Q\in L^p({\Bbb B}^n).$ Note that, $f$ is a homeomorphism in ${\Bbb
B}^n\setminus \{0\},$ and $f\in C^1\left({\Bbb B}^n\setminus
\{0\}\right).$ Now $f\in W_{loc}^{1, n}\left({\Bbb B}^n\setminus
\{0\}\right).$ Set
\begin{equation}\label{eq12C}
J(x, f)={\rm det}\, f^{\,\prime}(x),\quad
l\left(f^{\,\prime}(x)\right)\,=\,\,\,\min\limits_{h\in {\Bbb R}^n
\setminus \{0\}} \frac {|f^{\,\prime}(x)h|}{|h|}
\end{equation}
and
\begin{equation}\label{eq1.1.1}
K_{I}(x,f)\quad =\quad\left\{
\begin{array}{rr}
\frac{|J(x,f)|}{{l\left(f^{\,\prime}(x)\right)}^n}, & J(x,f)\ne 0,\\
1,  &  f^{\,\prime}(x)=0, \\
\infty, & \text{otherwise}
\end{array}
\right.\,.
\end{equation}
Then there exist systems of vectors $e_1,\ldots, e_n$ and
$\widetilde{e_1},\ldots,\widetilde{e_n},$ and nonnegative numbers
$\lambda_1(x_0),\ldots,\lambda_n(x_0),$
$\lambda_1(x_0)\leqslant\ldots\leqslant\lambda_n(x_0),$ such that
$f^{\,\prime}(x_0)e_i=\lambda_i(x_0)\widetilde{e_i}$
(see.~\cite[4.1.I]{Re}), and
$$|J(x_0, f)|=\lambda_1(x_0)\ldots\lambda_n(x_0),\quad
l(f^{\,\prime}(x_0)) =\lambda_1(x_0)\,, $$
%
$$K_I(x_0, f)=\frac{\lambda_1(x_0)\ldots\lambda_n(x_0)}{\lambda^n_1(x_0)}\,.$$
%
Since $f$ has a type $f(x)=\frac{x}{|x|}\rho(|x|),$ it is not
difficult to show that, the ''main vectors'' $e_{i_1},\ldots,
e_{i_n}$ and $\widetilde{e_{i_1}},\ldots, \widetilde{e_{i_n}}$ are
$(n-1)$ linearly independent tangent vectors to $S(0, r)$ at $x_0,$
where $|x_0|=r,$ and one radial vector, which is orthogonal to them.
We also can show that, in this case, the corresponding
''stretchings'', denoted as $\lambda_{\tau}(x_0)$ and $\lambda_r,$
are
$\lambda_{\tau}(x_0):=\lambda_{i_1}(x_0)=\ldots=\lambda_{i_{n-1}}(x_0)=\frac{\rho(r)}{r}$
и $\lambda_{r}(x_0):=\lambda_{i_n}=\rho^{\,\prime}(r),$
correspondingly. From other hand, it is known that $f$ is a ring
$Q$-homeomorphism at $x_0=0$ under $Q=K_I(x, f)$ (see
\cite[Theorem~8.6]{MRSY}).

Given $e\in {\Bbb S}^{n-1},$ observe that, $\frac{\partial
f}{\partial e}(x_0)=\lim\limits_{t\rightarrow
+0}\frac{f(x_0+te)-f(x_0)}{t}=\frac{\partial f}{\partial
e}(x_0)=f^{\,\prime}(x_0)e$ whenever $x_0$ is differentiability
point of $f.$ Let $\lambda_{\tau}(x_0)$ is a stretching,
corresponding to a tangent direction at $x_0\in {\Bbb
B}^n\setminus\{0\},$ and $\lambda_{r}(x_0)$ is a stretching,
corresponding to a radial direction at $x_0.$ Now
$$\lambda_{\tau}(x_0)=(1+|x_0|^{\alpha})/|x_0|\,,\qquad \lambda_{r}(x_0)=\alpha|x_0|^{\alpha-1}\,.$$
Since
$\lambda_{\tau}(x_0)\ge \lambda_{r}(x_0),$ we obtain that
$l(f^{\,\prime}(x_0))=\lambda_{r}(x_0).$ By (\ref{eq1A}), we have
that
\begin{equation}\label{eq2.7.2}
Q(x):=K_I(x_0,
f)=\left(\frac{1}{\alpha}\right)^{n-1}\cdot\frac{(1+|x_0|^{\alpha})^{n-1}}
{|x_0|^{\alpha(n-1)}}\,.
\end{equation}
For $r<1,$
$$Q(x)\le \frac{C}{|x|^{\alpha(n-1)}}\,,\quad C:=\left(\frac{2}{\alpha}\right)^{n-1}\,.$$
Thus, we obtain that $$\int\limits_{{\Bbb B}^n}\left(Q(x)\right)^p
dm(x)\le C^p \int\limits_{{\Bbb
B}^n}\frac{dm(x)}{|x|^{p\alpha(n-1)}}=$$
\begin{equation}\label{eq2.3A}=C^p\int\limits_0^1\int\limits_{S(0,
r)}\frac{d{\mathcal{A}}}{|x|^{p\alpha(n-1)}}dr=\omega_{n-1}C^p
\int\limits_0^1\frac{dr}{r^{(n-1)(p\alpha-1)}}\,.\end{equation}
Since $I:=\int\limits_0^1\frac{dr}{r^{\beta}}$ is convergent at
$\beta<1,$ the integral in right-hand side of (\ref{eq2.3A}) is
convergent, because $\beta:=(n-1)(p\alpha-1)$ satisfies $\beta<1$ at
$\alpha\in (0, n/p(n-1)).$

Now, $Q(x)\in L^p({\Bbb B}^n).$~$\Box$
\end{proof}

\medskip
{\bf Example 2.} Now we show that the $FMO$ condition can not be
replaced by a weaker requirement $Q\in L^p,$ $p\geqslant 1,$ in
Theorems \ref{theor4*!} and \ref{th1}. We consider the case
$X=X^{\,\prime}={\Bbb R}^n,$ also.

\medskip
Set $D:={\Bbb B}^n\setminus \{0\}\subset {\Bbb R}^n,$
$D^{\,\prime}:=B(0, 2)\setminus\{0\} \subset {\Bbb R}^n.$ Denote
$\frak{A}_{Q}$ a family of all ring $Q$-homeomorphisms $g:{\Bbb
B}^n\setminus \{0\}\rightarrow {\Bbb R}^n$ at $0.$ The following
statement holds.

\medskip
\begin{theorem}\label{th3.10.1}{\sl\, Given $p\ge 1,$ there exist
$Q:{\Bbb B}^n\rightarrow [1, \infty],$ $Q(x)\in L^p({\Bbb B}^n)$ and
$g_m\in \frak{A}_{Q}$ for which $g_m$ has a continuous extension to
$x_0=0,$ however, $\left\{g_m(x)\right\}_{m=1}^{\infty}$ is not
equicontinuous at $x_0=0.$}
\end{theorem}

\medskip
\begin{proof} Given $p\ge 1$ and $\alpha\in (0, n/p(n-1)),$
$\alpha<1,$ we define $g_m: {\Bbb B}^n\setminus\{0\}\rightarrow
{\Bbb R}^n$ as
$$ g_m(x)\,=\,\left
\{\begin{array}{rr} \frac{1+|x|^{\alpha}}{|x|}\cdot x\,, & 1/m\le|x|\le 1, \\
\frac{1+(1/m)^{\alpha}}{(1/m)}\cdot x\,, & 0<|x|< 1/m \ .
\end{array}\right.
$$
Observe that, $g_m$ maps $D={\Bbb B}^n\setminus \{0\}$ onto
$D^{\,\prime}=B(0,2)\setminus\{0\},$ and that $x_0=0$ is removable
singularity for $g_m,$ $m\in {\Bbb N}.$ Moreover,
$\lim\limits_{x\rightarrow 0}g_m(x)=0,$ and $g_m$ is a constant as
$|x|\ge 1/m.$  In fact, $g_m(x)\equiv g(x)$ for $x:\
\frac{1}{m}<|x|< 1,$ $m=1,2\ldots\,,$ where
$g(x)=\frac{1+|x|^{\alpha}}{|x|}\cdot x.$

Observe $g_m\in ACL({\Bbb B}^n).$ In fact,
$g_m^{(1)}(x)=\frac{1+(1/m)^{\alpha}}{(1/m)}\cdot x,$
$m=1,2,\ldots,$ belongs to $C^1$ in $B(0, 1/m+\varepsilon)$ at
sufficiently small $\varepsilon>0.$ From other hand,
$g_m^{(2)}(x)=\frac{1+|x|^{\alpha}}{|x|}\cdot x$ are $C^1$-mappings
in $$A(1/m-\varepsilon, 1, 0)=\left\{x\in {\Bbb R}^n:
1/m-\varepsilon<|x|<1\right\}$$ at small $\varepsilon>0.$ Thus $g_m$
are lipschitzian in ${\Bbb B}^n$ and, consequently, $g_m\in
ACL({\Bbb B}^n)$ (see, e.g., \cite[sect.~5, p.~12]{Va$_1$}). As
above, we obtain
$$K_I(x, g_m)=\left
\{\begin{array}{rr} \left(\frac{1+|x|^{\,\alpha}}{\alpha
|x|^{\,\alpha}}\right)^{n-1}\,, & 1/m\le|x|\le 1, \\
1\,,\qquad & 0<|x|< 1/m\,.
\end{array}\right.
$$
Observe that $K_I(x, g_m)\le c_m$ for every $m\in {\Bbb N}$ and some
constant. Now, $g_m\in W_{loc}^{1, n}({\Bbb B}^n)$ and $g_m^{-1}\in
W_{loc}^{1, n}(B(0, 2)),$ because $g_m$ and $g_m^{-1}$ are
quasiconformal (see, e.g., \cite[Corollary~13.3 and
Theorem~34.6]{Va$_1$}). By \cite[Theorem~8.6]{MRSY}, $g_m$ are ring
$Q$-homeomorphisms in $D={\Bbb B}^n\setminus \{0\}$ at $Q=Q_m(x):=
K_I(x, g_m).$
Moreover, $g_m$ are $Q$-homeomorphisms with
$Q=\left(\frac{1+|x|^{\,\alpha}}{\alpha
|x|^{\,\alpha}}\right)^{n-1}.$ Since $\alpha p(n-1)<n,$ we have
$Q\in L^p({\Bbb B}^n),$ see proof of the theorem \ref{theor14*!}.
From another hand, we have that
\begin{equation}\label{eq2!!!!!}
\lim\limits_{x\rightarrow 0} |g(x)|= 1\,,
\end{equation}
and $g$ maps ${\Bbb B}^n\setminus\{ 0\}$ onto $1<|y|< 2.$ By
(\ref{eq2!!!!!}), we obtain that
$$|g_m(x)|=|g(x)|\ge 1\qquad\qquad\forall\quad x:|x|\ge 1/m,\quad m=1,2,\ldots\,,$$
i.e. $\{g_m\}_{m=1}^{\infty}$ is not equicontinuous a the origin.
\end{proof}$\Box$


\medskip
{\bf Open problem 1.} If $X=X^{\,\prime}={\Bbb R}^n,$ for all
$\beta: [a,\,b)\rightarrow X^{\,\prime}$ and $x\in
f^{\,-1}\left(\beta(a)\right),$ an open discrete mapping $f$ has a
maximal $f$-lif\-ting starting at $x.$ {\it To describe properties
of the metric spaces  $(X, d, \mu)$ and $(X^{\,\prime},
d^{\,\prime}, \mu^{\,\prime}),$ for which, for every curve $\beta:
[a,\,b)\rightarrow X^{\,\prime}$ and $x\in
f^{\,-1}\left(\beta(a)\right),$ there exists a maximal $f$-lif\-ting
starting at $x$ under every open discrete mapping $f:X\rightarrow
X^{\,\prime}.$}

\medskip
{\bf Open problem 2.} We say that the path connected space $(X, d,
\mu)$ is weakly flat at a point $x_0\in X$ if, for every
neighborhood $U$ of the point $x_0$ and every number $P>0,$ there is
a neighborhood $V\subseteq U$ of $x_0$ such that
$M_{\alpha}(\Gamma(E, F, X))\geqslant P$ for any continua  $E$ and
$F$ in $X$ intersecting $\partial V$ and $\partial U$. We say that a
space $(X,d,\mu)$ is weakly flat, if it is weakly flat at every
point. {\it To find relationship between weakly flat spaces and
spaces, which are Ahlfors $\alpha$-regular and support
$(1;\alpha)$-Poincare inequality. }

{\bf \noindent Evgeny Sevost'yanov} \\
Zhitomir Ivan Franko State University,  \\
40 Bol'shaya Berdichevskaya Str., 10 008  Zhitomir, UKRAINE \\
Phone: +38 -- (066) -- 959 50 34, \\
Email: esevostyanov2009@mail.ru
\end{document}